\documentclass[submission]{FPSAC2021}

\usepackage{hyperref}
\usepackage{graphicx}
\usepackage{xcolor}
\usepackage{tikz}
\usetikzlibrary{patterns}

\usepackage{genyoungtabtikz}
\YFrench
\Yboxdim{12pt}

\newtheorem{thm}{Theorem}[section]

\newtheorem{prop}[thm]{Proposition}

\theoremstyle{definition}

\newtheorem{defn}[thm]{Definition}
\newtheorem{rem}[thm]{Remark}
\newtheorem{ex}[thm]{Example}

\DeclareMathOperator{\SYT}{SYT}
\DeclareMathOperator{\SSYT}{SSYT}

\newcommand{\ds}{\displaystyle}
\newcommand{\bb}[1]{\mathbb{#1}}

\newcommand{\mc}[1]{\mathcal{#1}}
\newcommand{\mb}[1]{\mathbf{#1}}

\newcommand{\tw}[1]{\widetilde{#1}}
\newcommand{\ov}[1]{\overline{#1}}
\newcommand{\la}{\lambda}
\newcommand{\vp}{\varphi}

\newcommand{\RS}{{\rm RS}}
\newcommand{\qrst}{q{\rm RS}t}

\newcommand{\col}{{\rm col}}
\newcommand{\row}{{\rm row}}
\newcommand{\id}{{\rm id}}

\newcommand{\U}{\mathcal{U}}
\newcommand{\D}{\mathcal{D}}

\newcommand{\Yred}{\Yfillcolour{red}}
\newcommand{\Ycyan}{\Yfillcolour{cyan}}
\newcommand{\Ywhite}{\Yfillcolour{white}}

\title[\qrst: Probabilistic Robinson--Schensted for Macdonald polynomials]{\qrst: A probabilistic Robinson--Schensted correspondence for Macdonald polynomials}

\author[F. Aigner and G. Frieden]{Florian Aigner\thanks{\href{mailto:florian.aigner@univie.ac.at}{florian.aigner@univie.ac.at}. FA was supported by the Austrian Science Fund FWF: Erwin Schr\"odinger Fellowship J 4387.}\addressmark{1} \and Gabriel Frieden\thanks{\href{mailto:gabriel.frieden@lacim.ca}{gabriel.frieden@lacim.ca}. GF was supported by a CRM-ISM postdoctoral fellowship.}\addressmark{1}}

\address{\addressmark{1} LaCIM, Universit\'e du Qu\'ebec \`a Montr\'eal, Montr\'eal, QC, Canada}

\received{\today}

\abstract{We present a probabilistic generalization of the Robinson--Schensted correspondence in which a permutation maps to several different pairs of standard Young tableaux with nonzero probability. The probabilities depend on two parameters $q$ and $t$, and the correspondence gives a new proof of the squarefree part of the Cauchy identity for Macdonald polynomials. By specializing $q$ and $t$ in various ways, one recovers both the row and column insertion versions of the Robinson--Schensted correspondence, as well as several $q$- and $t$-deformations of row and column insertion which have been introduced in recent years in connection with integrable probability.}


\begin{document}

\maketitle

\section{Introduction}
\label{sec: Intro}

The Robinson--Schensted (RS) correspondence is a bijection between permutations and pairs of standard Young tableaux of the same shape. This bijection, along with its generalization due to Knuth (RSK), has significant applications in combinatorics, representation theory, algebraic geometry, and probability. One of the most important features of RSK is that it gives a bijective proof of the Cauchy identity
\begin{equation}
\label{eq_intro_Cauchy}
\prod_{i,j \geq 1} \dfrac{1}{1-x_iy_j} = \sum_\la s_\la(\mathbf{x}) s_\la(\mathbf{y}),
\end{equation}
where the sum is over all partitions, and $s_\la(\mathbf{z})$ denotes a Schur function in the variables $\mathbf{z}=(z_1,z_2,\ldots)$. In particular, the RS case of RSK gives a bijective proof of the identity
\begin{equation}
\label{eq_intro_RS_Cauchy}
n! = \sum_{\la \vdash n} (f_\la)^2,
\end{equation}
where the sum is over all partitions of $n$, and $f_\la$ is the number of standard Young tableaux of shape $\la$; this identity arises from \eqref{eq_intro_Cauchy} by comparing the coefficients of the squarefree monomial $x_1 \cdots x_n y_1 \cdots y_n$ on either side.

In the past decade, several randomized versions of RS and RSK have been introduced \cite{BorodinPetrov16, BufetovMatveev18, BufetovPetrov15, MatveevPetrov17, OConnellPei13, Pei14}. In these versions, a permutation (or, for RSK, a nonnegative integer matrix) has nonzero probability of mapping to several different pairs of tableaux. The probabilities depend on a parameter $q$ or $t$ in $[0,1)$, and the algorithms give proofs of generalized Cauchy identities for $q$-Whittaker or Hall--Littlewood symmetric functions. These randomized insertion algorithms have applications to probabilistic models such as the TASEP, ASEP, and stochastic six-vertex model \cite{BorodinPetrov16, MatveevPetrov17, BufetovMatveev18}, and to the asymptotics of infinite matrices over a finite field \cite{BufetovPetrov15}.

In this abstract, we define a randomized generalization of RS which depends on two parameters $q$ and $t$. Our map is designed to give a new proof of the squarefree part of the Cauchy identity for the Macdonald symmetric functions $P_\la(\mb{x};q,t)$. The $P_\la(\mb{x};q,t)$ are ``master'' symmetric functions, in the sense that they specialize to many important families of symmetric functions (Schur, $q$-Whittaker, Hall--Littlewood, Jack). Similarly, our randomized algorithm, which we call $\qrst$, specializes to many of the known variants of RS, including the row and column insertion versions of ordinary RS, $q$-deformations of row and column insertion \cite{BorodinPetrov16, OConnellPei13, Pei14}, and a $t$-deformation of column insertion \cite{BufetovPetrov15}.

Another interesting specialization of $\qrst$ comes from setting $q = t$. This specialization reduces the Macdonald functions to the Schur functions, but it does not remove the randomness from our algorithm. Instead, it produces a one-parameter family of probabilistic insertion algorithms which interpolate between row insertion ($q=t \rightarrow 0$) and column insertion $(q=t \rightarrow \infty)$. At the intermediate value $q = t \rightarrow 1$, the probability that any $\sigma \in S_n$ inserts to a pair of standard Young tableaux of shape $\la$ is equal to the Plancherel measure of $\la$, and in fact each standard Young tableau of shape $\la$ appears as the insertion tableau with probability $f_\la/n!$. We also obtain a pair of identities involving hook-lengths and the numbers $f_\la$ (equations \eqref{eq_mu}, \eqref{eq_nu}), which we believe are new.

This extended abstract is organized as follows. In \S \ref{sec_RS}, we present the notion of an insertion algorithm from the perspective of up and down operators on Young's lattice and local growth rules. In \S \ref{sec_Mac}, we discuss Macdonald polynomials and introduce $(q,t)$-analogues of the up and down operators. In \S \ref{sec_qrst}, we present our probabilistic insertion algorithm, and in \S \ref{sec_properties}, we discuss some of its specializations. For further details, including proofs of our results, we refer the reader to our paper \cite{AF_qrst}.

\subsection*{Notation}
We assume the reader is familiar with (skew) Young diagrams, standard and semistandard Young tableaux (abbreviated $\SYT$ and $\SSYT$, respectively), and Schur functions, as defined, e.g., in \cite[Ch. 7]{EC2}. We draw Young diagrams in French notation. We write $\SYT(\la)$ (resp., $\SSYT(\la))$ for the set of standard (resp., semistandard) Young tableaux of shape $\la$. We call a SSYT with no repeated entries a \emph{partial standard Young tableau}. If $T$ is a SSYT, we denote by $T^{(i)}$ the shape of the subtableau consisting of entries at most $i$.


\section{Insertion algorithms via local growth rules}
\label{sec_RS}

\emph{Young's lattice} is the partial order $(\bb{Y}, \subseteq)$ on the set of partitions defined by inclusion of Young diagrams. For $\la,\mu \in \bb{Y}$, write $\mu \lessdot \la$ if $\mu \subseteq \la$ and $|\la/\mu| = 1$, and define
\[
\D(\lambda) = \{\mu \,|\, \mu \lessdot \lambda\}, \qquad \U(\lambda) = \{\nu \,|\, \nu \gtrdot \lambda\}.
\]
An \emph{inner corner} of $\la$ is a cell $c \in \la$ such that $\la/\mu = \{c\}$ for some $\mu \in \D(\la)$. An \emph{outer corner} of $\la$ is a cell $c \not \in \la$ such that $\nu/\la = \{c\}$ for some $\nu \in \U(\la)$. We will often identify the elements of $\D(\la)$ and $\U(\la)$ with the corresponding inner and outer corners of $\la$.

Let $\bb{Q} \bb{Y}$ be the $\bb{Q}$-vector space with basis $\bb{Y}$. The \emph{up operator} $U$ and \emph{down operator} $D$ are linear maps on $\bb{Q} \bb{Y}$ defined by $U \lambda = \sum_{\nu \in \U(\la)} \nu$ and $D \lambda = \sum_{\mu \in \D(\la)} \mu.$ These operators satisfy the commutation relation
\begin{equation}
\label{eq: up down commutator}
DU-UD = I,
\end{equation}
where $I$ is the identity map. This relation immediately implies the identity
\begin{equation}
\label{eq: n! = sum over SYT}
n! = \sum_{\la \vdash n} (f_\la)^2.
\end{equation}
Indeed, a standard Young tableau of shape $\la$ can be viewed as a saturated chain in Young's lattice from the empty partition to $\la$. This implies that the right-hand side of \eqref{eq: n! = sum over SYT} is equal to $\left\langle D^nU^n\emptyset,\emptyset \right\rangle$, where $\left\langle \cdot, \cdot \right\rangle$ is the inner product on $\bb{Q} \bb{Y}$ defined by $\left\langle \la, \mu \right\rangle= \delta_{\lambda,\mu}$ for $\la,\mu \in \bb{Y}$. On the other hand, a straightforward induction using the commutation relation \eqref{eq: up down commutator} shows that $\left\langle D^nU^n\emptyset,\emptyset \right\rangle$ is equal to $n!$.

The relation \eqref{eq: up down commutator} can be proved by reformulating it as the set of equations
\begin{align*}
|\U(\lambda)| &= |\D(\la)| + 1 & \text{ for all } \la, \\
|\U(\lambda) \cap \U(\rho)| &= |\D(\lambda) \cap \D(\rho)| & \text{ for } \la \neq \rho.
\end{align*}
It turns out to be quite fruitful to make the proofs of these equations explicitly bijective. For $\la \neq \rho$, this is uninteresting, since either $\D(\lambda) \cap \D(\rho) = \{\lambda \cap \rho\}$ and $\U(\lambda) \cap \U(\rho) = \{\lambda \cup \rho\}$, or both of these intersections are empty. For the equation $|\U(\la)| = |\D(\la)| + 1$, set
$
\D^*(\la) = \D(\la) \cup \{\la\},
$
and choose, for each $\la$, a bijection
\[
F_\lambda: \D^*(\lambda) \rightarrow \U(\lambda).
\]
Two choices for $F_\lambda$ are particularly natural: the \emph{row insertion bijection} $F_\la^{\textnormal{row}}$, and the \emph{column insertion bijection} $F_\la^{\textnormal{col}}$. The bijection $F_\la^\row$ sends $\la$ to the outer corner in the first row of $\la$, and the inner corner in row $i$ to the outer corner in row $i+1$; $F_\la^\col$ sends $\la$ to the outer corner in the first column of $\la$, and the inner corner in column $i$ to the outer corner in column $i+1$. Figure \ref{fig: map between down and up} illustrates these maps.

\begin{figure}
\begin{center}
\begin{tikzpicture}
\begin{scope}[scale=1]
\Ylinecolour{lightgray}
\tyng(0cm,0cm,7,5,5,2,1)
\Yred
\tgyoung(0cm,0cm,::::::;,,::::;,:;,;)
\Ywhite
\Ycyan
\tgyoung(0cm,0cm,:::::::;,:::::;,,::;,:;,;)
\Ywhite
\draw[line width=2pt] (0,0) -- (7*12pt,0) -- (7*12pt,12pt) -- (5*12pt,12pt) -- (5*12pt,3*12pt) -- (2*12pt,3*12pt) -- (2*12pt,4*12pt) -- (1*12pt,4*12pt) -- (1*12pt,5*12pt) -- (0*12pt,5*12pt) -- (0,0) --(7*12pt,0);
\draw[line width=1pt, ->] (.5*12pt,4.5*12pt) -- (.5*12pt,5.5*12pt);
\draw[line width=1pt, ->] (1.5*12pt,3.5*12pt) -- (1.5*12pt,4.5*12pt);
\draw[line width=1pt, ->] (4.5*12pt,2.5*12pt) -- (2.5*12pt,3.5*12pt);
\draw[line width=1pt, ->] (6.5*12pt,.5*12pt) -- (5.5*12pt,1.5*12pt);
\node at (3.5*12pt,-12pt) {$F_{\lambda}^{\textnormal{row}}$};

\begin{scope}[xshift=5cm]
\Ylinecolour{lightgray}
\tyng(0cm,0cm,7,5,5,2,1)
\Yred
\tgyoung(0cm,0cm,::::::;,,::::;,:;,;)
\Ywhite
\Ycyan
\tgyoung(0cm,0cm,:::::::;,:::::;,,::;,:;,;)
\Ywhite
\draw[line width=2pt] (0,0) -- (7*12pt,0) -- (7*12pt,12pt) -- (5*12pt,12pt) -- (5*12pt,3*12pt) -- (2*12pt,3*12pt) -- (2*12pt,4*12pt) -- (1*12pt,4*12pt) -- (1*12pt,5*12pt) -- (0*12pt,5*12pt) -- (0,0) --(7*12pt,0);
\draw[line width=1pt, ->] (.5*12pt,4.5*12pt) -- (1.5*12pt,4.5*12pt);
\draw[line width=1pt, ->] (1.5*12pt,3.5*12pt) -- (2.5*12pt,3.5*12pt);
\draw[line width=1pt, ->] (4.5*12pt,2.5*12pt) -- (5.5*12pt,1.5*12pt);
\draw[line width=1pt, ->] (6.5*12pt,.5*12pt) -- (7.5*12pt,.5*12pt);
\node at (3.5*12pt,-12pt) {$F_{\lambda}^{\textnormal{col}}$};
\end{scope}
\end{scope}
\end{tikzpicture}
\end{center}
\vspace{-3ex}
\caption{\label{fig: map between down and up} The Young diagram of the partition $\lambda=(7,5,5,2,1)$, with inner corners colored red and outer corners colored blue. The arrows depict the bijections $F_\lambda^{\textnormal{row}}$ and $F_\lambda^{\textnormal{col}}$. In both cases, the outer corner with no arrow pointing to it is the image of $\la$.}
\end{figure}
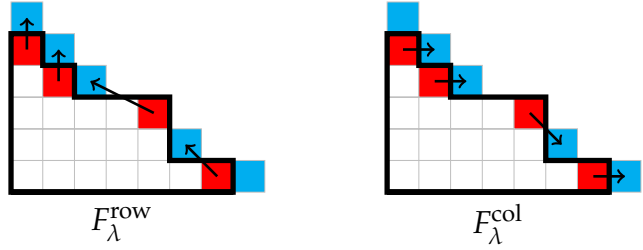

We call a collection of bijections $F_\bullet = \{F_\lambda \, | \, \la \in \bb{Y}\}$ a set of \emph{local growth rules}. Each set of local growth rules $F_\bullet$ determines a bijection
\[
\RS_{F_\bullet} : S_n \rightarrow \bigsqcup_{\la \vdash n} \SYT(\la) \times \SYT(\la).
\]
The bijection $\RS_{F_\bullet}$ is best understood using the formalism of Fomin's growth diagrams \cite{Fomin86}. This is explained in detail in \cite[\S 2.3]{AF_qrst}. However, since it would take too much space to introduce growth diagrams here, we instead describe $\RS_{F_\bullet}$ as an insertion algorithm.

\begin{defn}
\label{def_growth_insertion}
Let $F_\bullet$ be a set of local growth rules. Let $T$ be a partial standard Young tableau, and $k$ a number which is not an entry of $T$. Define the \emph{$F_\bullet$-insertion} of $k$ into $T$, denoted $T \xleftarrow{F_\bullet} k$, as follows:
\begin{itemize}
\item \textbf{Initial insertion step:} Place $k$ in the outer corner of $T^{(k)}$ corresponding to the partition $F_{T^{(k)}}(T^{(k)}) \in \mc{U}(T^{(k)})$. If this cell is occupied in $T$ by an entry $k' > k$, the entry $k'$ is displaced. Otherwise, the process terminates.
\item \textbf{Bumping step(s):} If an entry $z$ of $T$ is displaced by a smaller number, place $z$ in the outer corner of $T^{(z)}$ corresponding to $F_{T^{(z)}}(T^{(z-1)}) \in \mc{U}(T^{(z)})$. If $z$ displaces an entry $z' > z$, repeat this step for $z'$. Otherwise, the process terminates.
\end{itemize}

To compute $\RS_{F_\bullet} : \sigma \mapsto (P,Q)$, write $\sigma = \sigma_1 \cdots \sigma_n$ in one-line notation. Set $P_0 = \emptyset$, and recursively define
$
P_i = P_{i-1} \xleftarrow{F_\bullet} \sigma_i
$
for $i = 1, \ldots, n$. The \emph{insertion tableau} $P$ is the standard Young tableau $P_n$ obtained at the end of this process. The \emph{recording tableau} $Q$ is the standard Young tableau such that $Q^{(i)}$ is the shape of $P_i$ for each $i$; in other words, $Q$ contains an $i$ in the cell which was added to $P_{i-1}$ by the insertion of $\sigma_i$.
\end{defn}

The special case $\RS_{F_\bullet^{\row}}$ is the row insertion version of RS (as defined in, e.g., \cite[Ch. 7.11]{EC2}): each entry of $\sigma$ is initially inserted into the first row, and each displaced number is bumped to the next row. Similarly, $\RS_{F_\bullet^{\col}}$ is the column insertion version of RS. It follows easily from the perspective of growth diagrams that each $\RS_{F_\bullet}$ is a bijection, and moreover, that $\RS_{F_\bullet}(\sigma) = (P,Q) \iff \RS_{F_\bullet}(\sigma^{-1}) = (Q,P)$.


\section{Macdonald polynomials}
\label{sec_Mac}

\subsection{Monomial expansion of Macdonald polynomials}
\label{sec: Monomial expansion}

The \emph{Macdonald symmetric functions} $P_\la(\mb{x}; q,t)$ and $Q_\la(\mb{x}; q,t)$ are two families of symmetric functions\footnote{We refer to $P_\la$ and $Q_\la$ as \emph{Macdonald polynomials}, even though they are not actually polynomials.} in variables $\mb{x} = (x_1, x_2, \ldots)$, with coefficients in the field $\bb{Q}(q,t)$ of rational functions in $q$ and $t$. Both $P_\la(\mb{x};q,t)$ and $Q_\la(\mb{x};q,t)$ specialize to the Schur function $s_\la(\mb{x})$ when $q = t$. Macdonald originally defined the $P_\la$ and $Q_\la$ rather indirectly by a linear algebraic criterion, and then he derived explicit formulas for the monomial expansions of $P_\la$ and $Q_\la$ as weighted sums over semistandard Young tableaux of shape $\la$, generalizing the combinatorial formula for $s_\la$. We will take the somewhat unusual perspective of viewing these monomial expansions as the \emph{definition} of the Macdonald polynomials. To describe the expansions, we need some notation.

For a cell $c \in \lambda$, define its \emph{arm-length} $a_\lambda(c)$ (resp., \emph{leg-length} $\ell_\la(c)$) to be the number of cells in the Young diagram of $\la$ that are strictly to the right of (resp., strictly above) $c$, and define its \emph{hook-length} by $h_\lambda(c)=a_\lambda(c)+\ell_\lambda(c)+1$. For example, the Young diagram of $\la = (8,6,3,3,1)$ is shown below. The cell $c$ has $a_\la(c) = 6$, $\ell_\la(c) = 3$, and $h_\la(c) = 10$.

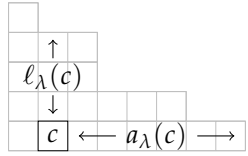
\footnotesize
\begin{center}
\begin{tikzpicture}[scale=0.7]
\Yboxdim{16 pt}
\Ylinecolour{lightgray}
\tyng(0,0,8,6,3,3,1)
\Ylinecolour{black}
\tgyoung(0pt,0pt,:;)
\node at (24pt,8pt) {$c$};
\node at (80pt,8pt) {$a_\lambda(c)$};
\draw[->] (102pt,8pt) -- (122pt,8pt);
\draw[->] (58pt,8pt) -- (38pt,8pt);
\node at (24pt,40pt) {$\ell_\lambda(c)$};
\draw[<-] (24pt,20pt) -- (24pt,30pt);
\draw[->] (24pt,50pt) -- (24pt,60pt);
\end{tikzpicture}
\end{center}
\normalsize

For $c \in \la$, define $b_\lambda(c) = \frac{[h_\la(c)]^\ell}{[h_\la(c)]^a}$, where $[h_\la(c)]^\ell = 1-q^{a_\la(c)}t^{\ell_\la(c)+1}$ and $[h_\la(c)]^a = 1-q^{a_\la(c)+1}t^{\ell_\la(c)}$ are two different $(q,t)$-analogues of the hook-length $h_\la(c)$. For $\mu \subseteq \la$, let $\mc{R}_{\la/\mu}$ (resp., $\mc{C}_{\la/\mu}$) be the set of all cells of $\mu$ which are in a row (resp., column) containing a cell of $\la/\mu$, and define\footnote{In \cite{Mac}, $\mc{R}_{\la/\mu}$ and $\mc{C}_{\la/\mu}$ are defined to include the cells in $\la/\mu$, so that $\vp_{\la/\mu}$ is just a product over $\mc{C}_{\la/\mu}$.}
\[
\psi_{\la/\mu}(q,t) = \prod_{c \in \mc{R}_{\la/\mu} - \, \mc{C}_{\la/\mu}} \dfrac{b_\mu(c)}{b_\la(c)}, \qquad\quad
\vp_{\la/\mu}(q,t) = \prod_{c \in \lambda/\mu}b_\lambda(c) \prod_{c \in \mc{C}_{\la/\mu}} \dfrac{b_\la(c)}{b_\mu(c)}.
\]
For a semistandard Young tableau $T$, define
\[
\psi_T(q,t) = \prod_{i \geq 1} \psi_{T^{(i)}/T^{(i-1)}}(q,t), \quad\qquad
\vp_T(q,t) = \prod_{i \geq 1} \vp_{T^{(i)}/T^{(i-1)}}(q,t),
\]
and let $\mb{x}^T = x_1^{\# \{1\text{'s in } T\}} x_2^{\# \{2\text{'s in } T\}} \cdots$.

\begin{thm}[{\cite[Ch. VI ($7.13,7.13'$)]{Mac}}]
\label{thm_Mac_monomial}
The Macdonald polynomials have monomial expansions
\[
P_\la (\mathbf{x}; q,t) = \sum_{T \in \SSYT(\la)} \psi_T(q,t) \mb{x}^T, \qquad\quad
Q_\la (\mathbf{x}; q,t) = \sum_{T \in \SSYT(\la)} \vp_T(q,t) \mb{x}^T.
\]
\end{thm}

\subsection{The generalized Cauchy identity}
\label{sec_qt_up_down}

Using the linear algebraic definition of the Macdonald polynomials, Macdonald proved the following generalization of the classical Cauchy identity \eqref{eq_intro_Cauchy}.

\begin{thm}[{\cite[Ch. VI (4.13)]{Mac}}]
\label{thm_Cauchy_Mac}
For $\mb{x} = (x_1, x_2, \ldots)$ and $\mb{y} = (y_1, y_2, \ldots)$, we have
\begin{equation}
\label{eq_Cauchy_Mac}
\prod_{i,j \geq 1} \dfrac{(1-tx_iy_j)(1-qtx_iy_j)(1-q^2tx_iy_j) \cdots}{(1-qx_iy_j)(1-q^2x_iy_j)(1-q^3x_iy_j) \cdots} = \sum_\la P_\la(\mb{x};q,t) Q_\la(\mb{y};q,t).
\end{equation}
\end{thm}

In this abstract we are interested, as in the discussion of the Schur case in \S \ref{sec_RS}, in the coefficients of the squarefree monomial $x_1 \cdots x_n y_1 \cdots y_n$ on either side of \eqref{eq_Cauchy_Mac}. Using the monomial expansions of $P_\la$ and $Q_\la$ on the right-hand side, we obtain the identity
\begin{equation}
\label{eq: squarefree Macdonald Cauchy}
\dfrac{(1-t)^n}{(1-q)^n} n! = \sum_{\la \vdash n} \sum_{P,Q} \psi_P(q,t) \vp_Q(q,t),
\end{equation}
where the inner sum is over pairs of standard Young tableaux of shape $\la$. Note that since $b_\la(c) = 1$ when $q = t$, this formula reduces to \eqref{eq_intro_RS_Cauchy} in the Schur specialization $q = t$.

The goal of this abstract is to give a direct proof of \eqref{eq: squarefree Macdonald Cauchy}, taking the monomial expansions of Theorem \ref{thm_Mac_monomial} as the definition of $P_\la$ and $Q_\la$. To this end, define $(q,t)$-analogues of the up and down operators on Young's lattice by
\[
U_{q,t} \la = \sum_{\nu \in \U(\la)} \psi_{\nu/\la}(q,t) \, \nu, \qquad\qquad
D_{q,t} \la = \sum_{\mu \in \D(\la)} \vp_{\la/\mu}(q,t) \, \mu.
\]
It is clear that the right-hand side of \eqref{eq: squarefree Macdonald Cauchy} is equal to $\langle D_{q,t}^nU_{q,t}^n \emptyset, \emptyset \rangle$. Thus, \eqref{eq: squarefree Macdonald Cauchy} can be deduced by induction on $n$ from the following commutation relation.

\begin{thm}
\label{thm: q,t commutator}
The $(q,t)$-up and down operators satisfy the commutation relation
\[
D_{q,t}U_{q,t}-U_{q,t}D_{q,t} = \frac{1-t}{1-q} I.
\]
\end{thm}

Reasoning as in \S \ref{sec_RS}, it is straightforward to reduce the proof of this commutation relation to the proof of the identity
\begin{equation}
\label{eq_la_la_qt}
\sum_{\nu \in \U(\la)} \psi_{\nu/\la}(q,t) \vp_{\nu/\la}(q,t) = \dfrac{1-t}{1-q} + \sum_{\mu \in \D(\la)} \psi_{\la/\mu}(q,t) \vp_{\la/\mu}(q,t)
\end{equation}
for each partition $\la$. In contrast to the situation in \S \ref{sec_RS}, however, \eqref{eq_la_la_qt} cannot be proved by a bijection $F_\la : \mc{D}^*(\la) \rightarrow \mc{U}(\la)$ (this can be seen in any example, the simplest of which is $\la = (1)$). Instead, we will ``probabilistically superimpose'' all the possible bijections $F_\la$.

\begin{rem}
It is possible to derive Theorem \ref{thm: q,t commutator} from the generalized Cauchy identity \eqref{eq_Cauchy_Mac}. Our philosophy, however, is that an explicit probabilistic proof of Theorem \ref{thm: q,t commutator}, and the resulting probabilistic insertion algorithm, is desirable for its own sake.
\end{rem}


\section{Definition of $\qrst$}
\label{sec_qrst}

\subsection{Probabilistic bijections}
\label{sec: probabilistic bijections}

The following definition is due to Bufetov and Petrov \cite{BufetovPetrov19}, although they use the name ``bijectivization.'' This notion also plays an important role in \cite{BufetovMatveev18}.
\begin{defn}
\label{def: prob bij}
Let $X$ and $Y$ be finite sets equipped with weight functions $\omega : X \rightarrow k$, $\ov{\omega} : Y \rightarrow k$, where $k$ is a field. A \emph{probabilistic bijection} from $(X,\omega)$ to $(Y,\ov{\omega})$ is a pair of maps $\mc{P},\ov{\mc{P}} : X \times Y \rightarrow k$ satisfying
\begin{enumerate}
\item For each $x \in X$, $\ds \sum_{y \in Y} \mc{P}(x,y) = 1$, and for each $y \in Y$, $\ds \sum_{x \in X} \ov{\mc{P}}(x,y) = 1$.
\item For each $x \in X$ and $y \in Y$, $\ds \omega(x)\mc{P}(x,y) = \ov{\mc{P}}(x,y)\ov{\omega}(y)$.
\end{enumerate}
We will write $\mc{P}(x \rightarrow y)$ for $\mc{P}(x,y)$ and $\ov{\mc{P}}(x \leftarrow y)$ for $\ov{\mc{P}}(x,y)$, and think of $\mc{P}(x \rightarrow y)$ as the probability of moving from $x$ to $y$, and $\ov{\mc{P}}(x \leftarrow y)$ as the probability of moving from $y$ to $x$. We find this terminology convenient even though we do not require that these expressions lie in $[0,1]$, or even that they be real-valued.
\end{defn}

It is easy to see that a probabilistic bijection from $(X,\omega)$ to $(Y,\ov{\omega})$ proves the identity
\[
\sum_{x \in X} \omega(x) = \sum_{y \in Y} \ov{\omega}(y).
\]

\subsection{A probabilistic bijection between $(\mc{D}^*(\la), \omega_\la)$ and $(\mc{U}(\la), \ov{\omega}_\la)$}
\label{sec_weighted_sets}

Define weight functions $\omega_\la : \mc{D}^*(\la) \rightarrow \bb{Q}(q,t)$ and $\ov{\omega}_\la : \mc{U}(\la) \rightarrow \bb{Q}(q,t)$ by
\[
\omega_\la(\mu) = \begin{cases} 
1 & \text{ if } \mu = \la \bigskip \\
\ds \prod_{c \in \mc{R}_{\la/\mu}} \dfrac{b_\mu(c)}{b_\la(c)} \prod_{c \in \mc{C}_{\la/\mu}} \dfrac{b_\la(c)}{b_\mu(c)} & \text{ if } \mu \in \mc{D}(\la),
\end{cases} \quad\quad
\ov{\omega}_\la(\nu) = \prod_{c \in \mc{R}_{\nu/\la}} \dfrac{b_\la(c)}{b_\nu(c)} \prod_{c \in \mc{C}_{\nu/\la}} \dfrac{b_\nu(c)}{b_\la(c)}.
\]
Using this notation, equation \eqref{eq_la_la_qt} becomes (after dividing both sides by $\frac{1-t}{1-q}$) 
\begin{equation}
\label{eq_weights}
\sum_{\mu \in \mc{D}^*(\la)} \omega_\la(\mu) = \sum_{\nu \in \mc{U}(\la)} \ov{\omega}_\la(\nu).
\end{equation}
We will prove \eqref{eq_weights} by giving a probabilistic bijection $\mc{P}_\la, \ov{\mc{P}}_\la$ from $(\mc{D}^*(\la), \omega_\la)$ to $(\mc{U}(\la), \ov{\omega}_\la)$.

The key to defining the probabilities $\mc{P}_\la$ and $\ov{\mc{P}}_\la$ is to split the weights $\omega_\la(\mu)$ and $\ov{\omega}_\la(\nu)$ into two pieces. For partitions $\rho \lessdot \kappa$, define
\[
\alpha_{\kappa/\rho} = \prod_{c \in \mc{R}_{\kappa/\rho}} \dfrac{[h_\rho(c)]^\ell}{[h_\kappa(c)]^\ell} \prod_{c \in \mc{C}_{\kappa/\rho}} \dfrac{[h_\rho(c)]^a}{[h_\kappa(c)]^a}, \qquad\quad \ov{\alpha}_{\kappa/\rho} = \prod_{c \in \mc{R}_{\kappa/\rho}} \dfrac{[h_\rho(c)]^a}{[h_\kappa(c)]^a} \prod_{c \in \mc{C}_{\kappa/\rho}} \dfrac{[h_\rho(c)]^\ell}{[h_\kappa(c)]^\ell}.
\]
Since $b_\kappa(c) = \frac{[h_\kappa(c)]^\ell}{[h_\kappa(c)]^a}$, we see immediately that for $\mu \in \mc{D}(\la)$ and $\nu \in \mc{U}(\la)$, we have
\begin{equation}
\label{eq_omega_alpha}
\omega_\la(\mu) = \dfrac{\alpha_{\la/\mu}}{\ov{\alpha}_{\la/\mu}}, \quad\quad \ov{\omega}_\la(\nu) = \dfrac{\alpha_{\nu/\la}}{\ov{\alpha}_{\nu/\la}}.
\end{equation}

\begin{defn}
\label{def_probs}
For $\mu \in \mc{D}^*(\la)$ and $\nu \in \mc{U}(\la)$, define
\[
\mc{P}_\la(\mu \rightarrow \nu) = \begin{cases}
t^{r_\nu-1} \alpha_{\nu/\la} & \text{ if } \mu = \la \medskip \\
t^{r_\nu-r_\mu-1} \dfrac{\alpha_{\nu/\la}}{\alpha_{\la/\mu}}\eta_{\nu/\la/\mu} & \text{ if } \mu \in \mc{D}(\la),
\end{cases}
\]
\[
\ov{\mc{P}}_\la(\mu \leftarrow \nu) = \begin{cases}
t^{r_\nu-1} \ov{\alpha}_{\nu/\la} & \text{ if } \mu = \la \medskip \\
t^{r_\nu-r_\mu-1} \dfrac{\ov{\alpha}_{\nu/\la}}{\ov{\alpha}_{\la/\mu}}\eta_{\nu/\la/\mu} & \text{ if } \mu \in \mc{D}(\la),
\end{cases}
\]
where
\[
\eta_{\nu/\la/\mu} = \dfrac{(1-q)(1-t)}{(1-q^{c_\mu-c_\nu}t^{r_\nu-r_\mu})(1-q^{c_\mu-c_\nu+1}t^{r_\nu-r_\mu-1})},
\]
and the cell $\nu/\la$ (resp., $\la/\mu$) is located in row $r_\nu$ and column $c_\nu$ (resp., row $r_\mu$ and column $c_\mu$).
\end{defn}

\begin{ex}
\label{ex_probs_d=1}
Suppose $\la = (h^v)$ is a rectangle of width $h$ and height $v$. In this case, $\mc{D}(\la)$ consists of the partition $\mu = (h^{v-1},h-1)$, and $\mc{U}(\la)$ consists of the two partitions $\nu_1 = (h+1,h^{v-1})$ and $\nu_2 = (h^v,1)$. One computes
\[
\begin{array}{ll}
\mc{P}_\la(\la \rightarrow \nu_1) = \dfrac{1-t^v}{1-q^ht^v} & \quad\quad \mc{P}_\la(\mu \rightarrow \nu_1) = qt^{v-1}\dfrac{1-q^{h-1}t}{1-q^ht^v} \bigskip \\
\mc{P}_\la(\la \rightarrow \nu_2) = t^v\dfrac{1-q^h}{1-q^ht^v} & \quad\quad \mc{P}_\la(\mu \rightarrow \nu_2) = \dfrac{1-qt^{v-1}}{1-q^ht^v}.
\end{array}
\]
\end{ex}

\begin{thm}
\label{thm_prob_bij}
\
\begin{enumerate}
\item For each $\mu \in \mc{D}^*(\la)$ (resp., $\nu \in \mc{U}(\la)$), we have
\[
\sum_{\nu \in \mc{U}(\la)} \mc{P}_\la(\mu \rightarrow \nu) = 1 \qquad\quad \text{(resp., } \sum_{\mu \in \mc{D}^*(\la)} \ov{\mc{P}}_\la(\mu \leftarrow \nu) = 1\text{)}.
\]
\item For $\mu \in \mc{D}^*(\la)$ and $\nu \in \mc{U}(\la)$, we have
$
\omega_\la(\mu) \mc{P}_\la(\mu \rightarrow \nu) = \ov{\mc{P}}_\la(\mu \leftarrow \nu) \ov{\omega}_\la(\nu).
$
\end{enumerate}
Thus, the expressions $\mc{P}_\la$ and $\ov{\mc{P}}_\la$ define a probabilistic bijection from $(\mc{D}^*(\la), \omega_\la)$ to $(\mc{U}(\la), \ov{\omega}_\la)$.
\end{thm}

The second part of this result follows immediately from \eqref{eq_omega_alpha}. The first part is proved by expressing the probabilities more explicitly in terms of a set of parameters associated to $\la$, and then using Lagrange interpolation (see \cite[\S 4.5]{AF_qrst} for details).

\begin{rem}
\label{rem_honest_probs}
The expressions $\mc{P}_\la(\mu \rightarrow \nu)$ and $\ov{\mc{P}}_\la(\mu \leftarrow \nu)$ take values in $[0,1]$ when $q,t \in [0,1)$ or $q,t \in (1,\infty)$. We believe this provides justification for calling these expressions probabilities.
\end{rem}

\subsection{Probabilistic insertion}

We now view the probabilities $\mc{P}_\la$ as a set of (probabilistic) local growth rules, and define the $\qrst$ algorithm analogously to the deterministic insertion algorithms $\RS_{F_\bullet}$ in \S \ref{sec_RS}.

\begin{defn}
Let $T$ be a partial standard Young tableau, and $k$ a number which is not an entry of $T$. The \emph{$(q,t)$-Robinson--Schensted insertion} of $k$ into $T$, denoted $T \xleftarrow{\qrst} k$,
is a probability distribution on partial SYTs, which is computed as follows:
\begin{itemize}
\item \textbf{Initial insertion step:} For each $\nu \in \U(T^{(k)})$, place $k$ in the cell $\nu/T^{(k)}$ with probability $\mc{P}_{T^{(k)}}(T^{(k)} \rightarrow \nu)$.

\item \textbf{Bumping step(s):} Suppose an entry $z$ of $T$ is displaced by a smaller number. For each $\nu \in \mc{U}(T^{(z)})$, place $z$ in the cell $\nu/T^{(z)}$ with probability $\mc{P}_{T^{(z)}}(T^{(z-1)} \rightarrow \nu)$.
\end{itemize}
In other words, the probability that $(T \xleftarrow{\qrst} k) = T'$ is the sum of the probabilities of all ``insertion paths'' that produce $T'$.

The \emph{$(q,t)$-Robinson--Schensted (\qrst) correspondence} associates to each $\sigma \in S_n$ a probability distribution $\mc{P}_{\qrst}(\sigma \rightarrow P,Q)$ on pairs of SYTs of the same shape, where $\mc{P}_{\qrst}(\sigma \rightarrow P,Q)$ is the sum of the probabilities of all ways of successively inserting $\sigma_1, \ldots, \sigma_n$ using the above procedure, starting with the empty tableau, such that the end result is $P$, and the growth at each step is recorded by $Q$.
\end{defn}

As in the case of the bijections $\RS_{F_\bullet}$, Fomin's growth diagrams give an elegant way of defining the probabilities $\mc{P}_{\qrst}$, as well as the ``backward'' or ``inverse'' probabilities $\ov{\mc{P}}_{\qrst}$. This perspective leads to a straightforward proof that $\mc{P}_{\qrst}, \ov{\mc{P}}_{\qrst}$ give a probabilistic bijection between the weighted sets of permutations and pairs of standard Young tableaux of the same shape (with weight functions $\frac{(1-t)^n}{(1-q)^n}$ and $\psi_P(q,t)\vp_Q(q,t)$, respectively). The growth diagram point of view also makes it clear that $\qrst$ enjoys the symmetry property
\[
\mc{P}_{\qrst}(\sigma \rightarrow P,Q) = \mc{P}_{\qrst}(\sigma^{-1} \rightarrow Q,P).
\]
The details appear in \cite[\S 4.6]{AF_qrst}.

\begin{ex}
\label{ex_qrst_insertion}
We compute the probability distribution $\mc{P}_{\qrst}(\sigma \rightarrow P,Q)$ for $\sigma = 312 \in S_3$. The insertion of 3 into the empty tableau produces $\young(3)$. When 1 is inserted into $\young(3)$, the 1 displaces the 3, and the bumping of the 3 produces
\[
\begin{array}{llc}
\young(13) & \text{ with probability } & \mc{P}_{(1)}(\emptyset \rightarrow (2)) = \frac{q(1-t)}{1-qt} \bigskip \\
\young(1,3) & \text{ with probability } & \mc{P}_{(1)}(\emptyset \rightarrow (1,1)) = \frac{1-q}{1-qt}.
\end{array}
\]
(The expressions for the probabilities come from Example \ref{ex_probs_d=1}.) When 2 is inserted into $\young(13)$, it either displaces the 3, or goes into the second row. In the former case, the 3 either remains in the first row, or moves to the second row. This results in
\[
\begin{array}{llc}
\young(123) & \text{ with probability } & \mc{P}_{(1)}((1) \rightarrow (2)) \mc{P}_{(2)}((1) \rightarrow (3)) = \frac{1-t}{1-qt} \frac{q(1-qt)}{1-q^2t} \bigskip \\
\young(12,3) & \text{ with probability } & \mc{P}_{(1)}((1) \rightarrow (2)) \mc{P}_{(2)}((1) \rightarrow (2,1)) = \frac{1-t}{1-qt} \frac{1-q}{1-q^2t} \bigskip \\
\young(13,2) & \text{ with probability } & \mc{P}_{(1)}((1) \rightarrow (1,1)) = \frac{t(1-q)}{1-qt}.
\end{array}
\]
The insertion of 2 into $\young(1,3)$ is computed similarly. The end result is the probability distribution shown below.
\[
\begin{array}{ccccccc}
P\quad & \young(123) & \young(12,3) & \young(13,2) & \young(12,3) & \young(13,2) & \young(1,2,3) \medskip \\
Q\quad & \young(123) & \young(12,3) & \young(12,3) & \young(13,2) & \young(13,2) & \young(1,2,3) \bigskip \\
\quad\quad\quad\quad & \frac{q^2(1-t)^2}{(1-qt)(1-q^2t)} & \frac{q(1-q)(1-t)^2}{(1-qt)^2(1-q^2t)} & \frac{qt(1-q)(1-t)}{(1-qt)^2} & \frac{(1-q)(1-t)}{(1-qt)^2} & \frac{qt^2(1-q)^2(1-t)}{(1-qt)^2(1-qt^2)} & \frac{t(1-q)^2}{(1-qt)(1-qt^2)}
\end{array}
\]
\end{ex}


\section{Specializations of $\qrst$}
\label{sec_properties}

\begin{figure}
\begin{center}
\begin{tikzpicture}

\draw (7,7) node{\textcolor{blue}{$\qrst$}};

\footnotesize
\draw (1.5,5.19) node{\begin{tabular}{c} \color{red}{$t$-RS} \\ \color{red}{(row insertion)} \end{tabular}};
\draw (4.5,5) node{\begin{tabular}{c} \color{cyan}{$q$-RS} \\ \color{cyan}{(row insertion)} \\ \cite{BorodinPetrov16, MatveevPetrov17} \end{tabular}};
\draw (9.5,5) node{\begin{tabular}{c} \color{cyan}{$q$-RS} \\ \color{cyan}{(column insertion)} \\ \cite{OConnellPei13, Pei14, MatveevPetrov17} \end{tabular}};
\draw (12.5,5) node{\begin{tabular}{c} \color{red}{$t$-RS} \\ \color{red}{(column insertion)} \\ \cite{BufetovPetrov15, BufetovMatveev18} \end{tabular}};

\draw[->] (7,6.7) --node[right]{$q = t$} (7,4);
\draw[->] (7.2,6.7) --node[right]{$\;\; t \rightarrow \infty, q \rightarrow q^{-1}$} (8.7,5.7);
\draw[->] (7.5,6.7) -- (11,6.7) --node[right]{$\;\; q \rightarrow \infty, t \rightarrow t^{-1}$} (12,5.7);
\draw[->] (6.8,6.7) --node[left]{$t \rightarrow 0 \;\;$} (5.3,5.7);
\draw[->] (6.5,6.7) -- (3,6.7) --node[left]{$q \rightarrow 0 \;\;$} (2,5.7);

\draw[->] (7,3) --node[right]{$q \rightarrow 1$} (7,2.1);
\draw[->] (6.8,3) --node[above]{$q \rightarrow 0 \;\;$} (4,2.1);
\draw[->] (7.2,3) --node[above]{$\quad\;\; q \rightarrow \infty$} (10,2.1);

\draw[->] (1.5,4.3) --node[left]{$t \rightarrow 0$} (3,2.1);
\draw[->] (4.5,4.3) --node[left]{$q \rightarrow 0$} (3.5,2.1);
\draw[->] (9.5,4.3) --node[right]{$q \rightarrow 0$} (10.5,2.1);
\draw[->] (12.5,4.3) --node[right]{$t \rightarrow 0$} (11,2.1);

\draw (3,1.8) node{\textcolor{purple}{RS (row insertion)}};
\draw (11,1.8) node{\textcolor{purple}{RS (column insertion)}};

\draw (7,3.5) node{\textcolor{purple}{\begin{tabular}{c} $q$-Plancherel measure \\ (for $\sigma = \id$) \end{tabular}}};
\draw (7,1.61) node{\textcolor{purple}{\begin{tabular}{c} Plancherel measure \\ (for all permutations $\sigma$) \end{tabular}}};
\normalsize

\end{tikzpicture}
\end{center}
\vspace{-2ex}
\caption{Specializations of $\qrst$. The color indicates the corresponding specialization of the Macdonald functions: $q$-Whittaker functions for the $q$-RS insertions; Hall--Littlewood functions for the $t$-RS insertions; Schur functions for the others.}
\label{fig_spec_chart}
\end{figure}
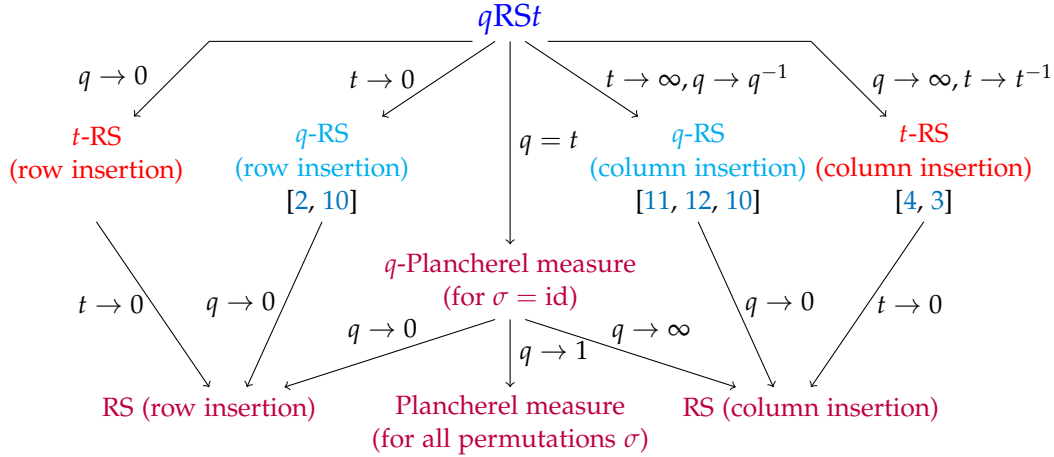

Figure \ref{fig_spec_chart} summarizes a number of specializations of $\qrst$. The reader may investigate these in Example \ref{ex_qrst_insertion}. Here we will focus on the $q = t$ specialization. This specialization gives rise to a one-parameter family of probabilistic bijections between the (trivially) weighted sets $(S_n, 1)$ and $(\bigsqcup_{\la \vdash n} \SYT(\la) \times \SYT(\la), 1)$, which contains both the row and column insertion versions of RS ($q = t = 0$ and $q = t \rightarrow \infty$, respectively). By Remark \ref{rem_honest_probs}, each nonnegative value of the parameter gives rise to a probabilistic bijection consisting of actual probabilities. The value $q = t \rightarrow 1$ is particularly interesting.

\begin{prop}
\label{prop_q=t=1}
Let $\tw{\mc{P}}_\la$ denote the $q = t \rightarrow 1$ specialization of $\mc{P}_\la$ (and of $\ov{\mc{P}}_\la$). We have
\[
\tw{\mc{P}}_\la(\la \rightarrow \nu) = \dfrac{H_\la}{H_\nu}, \qquad\qquad \tw{\mc{P}}_\la(\mu \rightarrow \nu) = \dfrac{H_\la^2}{H_\mu H_\nu} \dfrac{1}{h_\la(c_{\mu,\nu})^2}
\]
for $\mu \in \mc{D}(\la)$ and $\nu \in \mc{U}(\la)$. Here $H_\kappa = \prod_{c \in \kappa} h_\kappa(c)$ is the product of the hook-lengths of $\kappa$, and $c_{\mu,\nu}$ is the unique cell in $\la$ which is in $\mc{R}_{\nu/\la} \cap \mc{C}_{\la/\mu}$ or $\mc{C}_{\nu/\la} \cap \mc{R}_{\la/\mu}$.
\end{prop}

\begin{proof}
In the limit $q = t \rightarrow 1$, we have $\alpha_{\kappa/\rho} \rightarrow H_\rho/H_\kappa$ and $\eta_{\nu/\la/\mu} \rightarrow 1/h_\la(c_{\mu,\nu})^2$.
\end{proof}

By substituting this result into Theorem \ref{thm_prob_bij}(1) and using the hook-length formula $f_\la = n!/H_\la$, we obtain the following identities for $\la \vdash n$:
\begin{align}
\label{eq_upper} \sum_{\nu \in \mc{U}(\la)} f_\nu &= (n+1) f_\la \\
\label{eq_mu} \sum_{\nu \in \mc{U}(\la)} \dfrac{f_\mu f_\nu}{(h_\la(c_{\mu,\nu}))^2} &= \dfrac{n+1}{n} (f_\la)^2 \quad\quad \text{ for } \mu \in \mc{D}(\la) \\
\label{eq_nu} \dfrac{f_\la f_\nu}{n} + \sum_{\mu \in \mc{D}(\la)} \dfrac{f_\mu f_\nu}{(h_\la(c_{\mu,\nu}))^2} &= \dfrac{n+1}{n} (f_\la)^2 \quad\quad \text{ for } \nu \in \mc{U}(\la).
\end{align}

The identity \eqref{eq_upper} is a classical result known as the ``upper recursion'' for the numbers $f_\la$. Greene, Nijenhuis, and Wilf showed that the ratios $H_\la/H_\nu$ arise from a ``random hook walk'' taking place outside the Young diagram of $\la$, thereby giving a beautiful explanation for why the ratios $H_\la/H_\nu$ define a probability distribution on $\mc{U}(\la)$ \cite{GNW2}. We show in \cite[\S 6]{AF_qrst} that the probabilities $\mc{P}_\la(\la \rightarrow \nu)$ arise from a $(q,t)$-generalization of this random hook walk (this was inspired by the $(q,t)$-hook walk in \cite{GarHai}). The identities \eqref{eq_mu} and \eqref{eq_nu}, on the other hand, seem to be new, and we believe they deserve further study.

We end with a result whose proof will appear in forthcoming work of the authors (the $\sigma = \id$ case follows easily from Proposition \ref{prop_q=t=1}).

\begin{prop}
Suppose $\sigma \in S_n$ and $P \in \SYT(\la)$ for some $\la \vdash n$. In the $q = t \rightarrow 1$ specialization of $\qrst$, the probability that $P$ is the insertion tableau of $\sigma$ is equal to $f_\la/n!$.
\end{prop}

\acknowledgements{This project grew out of a working group at LaCIM during 2019-2020. We are grateful to all the members of the group, and especially to Hugh Thomas, Fran\c cois Bergeron, and Steven Karp, for many interesting discussions.}

\bibliographystyle{plain}
\bibliography{Aigner_Frieden_qRSt}

\end{document}